\documentclass{article}

\usepackage{amsmath,amsthm,amssymb,mathtools}  
\usepackage{dsfont}				
\usepackage{enumerate}
\usepackage{graphicx}			
\usepackage[caption=false]{subfig} 

\usepackage[T1]{fontenc}	

\usepackage{bm}					

\usepackage[hyperindex,breaklinks]{hyperref}		
\hypersetup{colorlinks=true, linkcolor=blue, citecolor=blue}

\usepackage[nameinlink]{cleveref} 	

\usepackage[latin1]{inputenc}                       
\usepackage[disable,
			colorinlistoftodos,textsize=scriptsize]{todonotes}		
\newcommand{\SELF}[1]{\todo[color=green!50]{#1}} 

\newtheorem{theorem}{Theorem}[section]
\newtheorem{proposition}[theorem]{Proposition}
\newtheorem{corollary}[theorem]{Corollary}
\newtheorem{lemma}[theorem]{Lemma}

\theoremstyle{definition}
\newtheorem*{definition}{Definition} 

\newtheorem{example}[theorem]{Example}

\newcommand\inner[1] 		{\langle #1 \rangle}		
\def\lcontr		{\lrcorner\,}				
\def\glcontr	{\rfloor}					
\def\Proj			{\mathrm{Proj}}

\def\N{\mathds{N}}

\def\R{\mathds{R}}
\def\C{\mathds{C}}

\def\I{\mathcal{I}}

\def\im	 				{\mathrm{i}}

\newcommand{\pperp}{\mathrel{\ooalign{\hss$\perp$\hss\cr%
  \kern0.11ex\raise0.45ex\hbox{\rotatebox{-5}{\scalebox{0.66}{$\slash$}}}}}}

\DeclareMathOperator{\Span}{span}

\DeclareMathOperator{\vol}{vol}

\begin{document}

\title{Grassmann angle formulas and identities}

\author{Andr\'e L. G. Mandolesi 
               \thanks{Instituto de Matemática e Estatística, Universidade Federal da Bahia, Av. Adhemar de Barros s/n, 40170-110, Salvador - BA, Brazil. ORCID 0000-0002-5329-7034. E-mail: \texttt{andre.mandolesi@ufba.br}}}
               
\date{\today
\SELF{v2.1}
}

\maketitle

\begin{abstract}
	Grassmann angles improve upon similar concepts of angle between subspaces that measure volume contraction in orthogonal projections, working for real or complex subspaces, and being more efficient when dimensions are different. Their relations with contractions, inner and exterior products of multivectors are used to obtain formulas for computing these or similar angles in terms of arbitrary bases, and various identities for the angles with certain families of subspaces. These include generalizations of the Pythagorean trigonometric identity $\cos^2\theta+\sin^2\theta=1$ for high dimensional and complex subspaces, which are connected to generalized Pythagorean theorems for volumes, quantum probabilities and Clifford geometric product.

	\vspace{.5em}
	
	\vspace{3pt}
	
\end{abstract}

\section{Introduction}

Measuring the separation between subspaces is important in many areas, from geometry and linear algebra to statistics \cite{Hotelling1936}, operator perturbation theory \cite{Kato1995}, data mining \cite{Jiao2018}, etc.
In high dimensions, no single number fully describes this separation, and different concepts are used: gap, principal angles, minimal angle, Friedrichs angle, and others.
An often used angle concept \cite{Gluck1967,Gunawan2005,Hitzer2010a,Jiang1996} combines principal angles in a way that shows how volumes contract when orthogonally projected between the subspaces. 
Research on it has been focused on real spaces, but many applications require complex ones.

The relation between angle and volume contraction changes significantly in the complex case. Taking this into account, in \cite{Mandolesi_Grassmann} we defined a Grassmann angle that works well in both cases, with only a few adjustments.
It is based on that same angle, but has a subtly important difference, which makes it easier to use with subspaces of distinct dimensions: it is asymmetric, reflecting the dimensional asymmetry between the subspaces in a way that leads to better properties and more general results.
This angle is intrinsically connected with Grassmann algebra, and has found interesting applications in the  geometry of Grassmannians, Clifford algebra, quantum theory, etc. 

In \cite{Mandolesi_Products} we expressed contractions, inner, exterior and Clifford products of blades (simple multivectors) in terms of Grassmann angles.
In this article we use the products to get formulas for the angles in terms of arbitrary bases, and identities relating the angles with the subspaces of an orthogonal partition, with coordinate subspaces of orthogonal bases, and others. 
Some of them generalize the Pythagorean trigonometric identity $\cos^2\theta+\sin^2\theta=1$ for high dimensional and complex subspaces, being related to real and complex Pythagorean theorems for volumes \cite{Mandolesi_Pythagorean} and quantum probabilities \cite{Mandolesi_Born}, and giving a geometric interpretation for an algebraic property of the Clifford product.
Some of our results correspond, in the real case, to known ones, but our methods provide simpler proofs, while also extending them to the complex case.

\Cref{sc:preliminaries} presents concepts and results which will be needed.
We obtain formulas for computing Grassmann angles in \cref{sc:Formulas}, generalized Pythagorean identities in \cref{sc:Pythagorean trigonometric}, and other useful identities in \cref{sc:other identities}.

\section{Preliminaries}\label{sc:preliminaries} 

Here we review results we will use. See \cite{Mandolesi_Grassmann,Mandolesi_Products} for proofs and more details.

In this article, $X$ is a $n$-dimensional vector space over $\R$ (real case) or $\C$ (complex case), with inner product  $\inner{\cdot,\cdot}$ (Hermitian product in the complex case, with conjugate-linearity in the first argument).
For subspaces $V,W\subset X$, $\Proj_W:X\rightarrow W$ and $\Proj^V_W:V\rightarrow W$ are orthogonal projections.
A \emph{line} is a 1-dimensional subspace. 

\subsection{Grassmann algebra and partial orthogonality}

In the Grassmann algebra $\Lambda X$ \cite{Marcus1975,Yokonuma1992}, a \emph{$p$-blade} is a simple multivector $\nu=v_1\wedge\ldots\wedge v_p \in\Lambda^p X$ of \emph{grade} $p$, where $v_1,\ldots,v_p\in X$. If $\nu\neq 0$, it \emph{represents} the $p$-dimensional subspace $V=\Span(v_1,\ldots,v_p)$, and $\Lambda^p V=\Span(\nu)$.
A scalar $\nu\in\Lambda^0 X$ is a $0$-blade, representing $\{0\}$. 

The inner product of $\nu=v_1\wedge\ldots\wedge v_p$ and $\omega=w_1\wedge\ldots\wedge w_p$,
\begin{equation*}
\inner{\nu,\omega} = \det\!\big(\inner{v_i,w_j}\big) = \begin{vmatrix}
\inner{v_1,w_1}& \cdots & \inner{v_1,w_p} \\ 
\vdots & \ddots & \vdots \\ 
\inner{v_p,w_1} & \cdots & \inner{v_p,w_p}
\end{vmatrix},
\end{equation*}
is extended linearly (sesquilinearly, in the complex case), with blades of distinct grades being  orthogonal, and $\inner{\nu,\omega}=\bar{\nu}\omega$ for  $\nu,\omega\in\Lambda^0 X$.
In the real case, the norm $\|\nu\|=\sqrt{\inner{\nu,\nu}}$ gives the $p$-dimensional volume of the parallelotope spanned by $v_1,\ldots,v_p$. In the complex case, $\|\nu\|^2$ gives the $2p$-dimensional volume of the parallelotope spanned by $v_1,\im v_1,\ldots,v_p, \im v_p$. 

Given a subspace $W\subset X$, let $P=\Proj_W$. 
The orthogonal projection of $\nu=v_1\wedge\ldots\wedge v_p$  on $\Lambda W\subset\Lambda X$ is $P\nu = Pv_1\wedge\ldots\wedge Pv_p$.

Orthogonality in $\Lambda^p X$ corresponds to a weaker orthogonality concept in $X$.

\begin{definition}
For subspaces $V,W\subset X$, $V$ is \emph{partially orthogonal} to $W$ ($V\pperp W$) if there is a nonzero $v\in V$ such that $\inner{v,w}=0$ for all $w\in W$. 
\end{definition}

\begin{proposition}\label{pr:partial orthogonality}
Let $V,W\subset X$ and $U\subset V$ be subspaces, with $V$ represented by a blade $\nu$, and $P=\Proj_W$. Then:
\begin{enumerate}[i)]
\item $V \pperp W$ $\Leftrightarrow$ $\dim P(V)<\dim V$. \label{it:orthogonality}
\item If $V\pperp W$ then $P\nu=0$, otherwise $P\nu$ represents $P(V)$. \label{it:Pnu represents PV}
\item If $V\not\pperp W$ then $U\not\pperp W$. \label{pr:subspace not pperp}
\end{enumerate}
\end{proposition}

\begin{proposition}\label{pr:partial orth Lambda orth}
Let $V,W\subset X$ be nonzero subspaces, and $p=\dim V$. Then $V\pperp W \Leftrightarrow \Lambda^p V \perp \Lambda^p W$.
\end{proposition}

\subsection{Coordinate decomposition}\label{sc:Coordinate decomposition}

\begin{definition}
For integers $1\leq p\leq q$, let 
\[ \I_p^q= \{ (i_1,\ldots,i_p)\in\N^p : 1\leq i_1 < \ldots<i_p\leq q \}. \]
For any multi-index $I=(i_1,\ldots,i_p)\in\I_p^q$, we write $|I|=i_1+\ldots+i_p$ and, if $p<q$, $\hat{I}=(1,\ldots,\hat{i_1},\ldots,\hat{i_p},\ldots,q)\in\I_{q-p}^q$, where each $\hat{i_k}$ indicates that index has been removed.
Also, let $\I_0^q=\{0\}$, and for $I\in\I_0^q$ let $|I|=0$ and $\hat{I}=(1,\ldots,q)\in\I_q^q$. 
For $I\in\I_q^q$ let $\hat{I}=0\in\I_0^q$.
\end{definition}

\begin{definition}
	Given a basis $\beta=(w_1,\ldots,w_q)$ of $W\subset X$, and $1\leq p\leq q$, the \emph{coordinate $p$-subspaces} for $\beta$ are the $\binom{q}{p}$ subspaces  
	\begin{equation*}
		W_I = \Span(w_{ i_1},\ldots,w_{ i_p}),
	\end{equation*}
	for $I=(i_1,\ldots,i_p)\in\I_p^q$, which are represented by the \emph{coordinate $p$-blades} of $\beta$,
	\begin{equation}\label{eq:coordinate blades}
	\omega_I = w_{ i_1}\wedge\ldots\wedge w_{ i_p}\in \Lambda^p W_I.
	\end{equation}
	For $I\in\I_0^q$ we have the coordinate $0$-subspace $W_I=\{0\}$ and $\omega_I=1\in\Lambda^0 W_I$.
\end{definition}

When $\beta$ is orthonormal, $\{\omega_I\}_{I\in\I_p^q}$ is an orthonormal basis of $\Lambda^p W$.

\begin{definition}
Given a decomposed nonzero blade $\omega= w_1\wedge\ldots\wedge w_q\in\Lambda^q X$, take the basis $\beta=(w_1,\ldots,w_q)$ of its subspace. For any $0\leq p\leq q$ and $I\in \I_p^q$, the \emph{coordinate decomposition} of $\omega$ (w.r.t. $I$ and $\beta$) is
\begin{equation}\label{eq:multiindex decomposition}
\omega = \sigma_I \,\omega_I\wedge\omega_{\hat{I}},
\end{equation}
with $\omega_I$ and $\omega_{\hat{I}}$ as in \eqref{eq:coordinate blades}, and $\sigma_I = (-1)^{|I|+\frac{p(p+1)}{2}}$. 
\SELF{$=\sigma_{\omega_I\wedge\omega_{\hat{I}},\omega}$. Tanto faz ser $(-1)^{|I|\pm \frac{p(p+1)}{2}}$}
\end{definition}

%
%

\subsection{Principal angles and vectors}

\begin{definition}
	The \emph{Euclidean angle} $\theta_{v,w}\in[0,\pi]$ between nonzero vectors $v,w\in X$ is given by
	$\cos\theta_{v,w} = \frac{\operatorname{Re}\inner{v,w}}{\|v\| \|w\|}$. In the complex case, there is also a \emph{Hermitian angle} $\gamma_{v,w}\in[0,\frac \pi 2]$ defined  by $\cos\gamma_{v,w} = \frac{|\langle v,w \rangle|}{\|v\| \|w\|}$. 
\end{definition}

The Hermitian angle is the Euclidean angle between $v$ and $\Span_{\C}(w)$.

A list of principal angles \cite{Galantai2006,Jordan1875} is necessary to fully describe the relative position of high dimensional subspaces.

\begin{definition}
	Let $V,W\subset X$ be nonzero subspaces, $p=\dim V$, $q=\dim W$ and $m=\min\{p,q\}$.
	Orthonormal bases $(e_1,\ldots,e_p)$ of $V$ and $(f_1,\ldots,f_q)$ of $W$ are associated \emph{principal bases}, formed by \emph{principal vectors}, with \emph{principal angles} $0\leq \theta_1\leq\ldots\leq\theta_m\leq\frac \pi 2$, if
	\begin{equation}\label{eq:inner ei fj}
		\inner{e_i,f_j} = \delta_{ij}\cos\theta_i.
	\end{equation}
\end{definition}

A singular value decomposition \cite{Galantai2006,Golub2013} gives such bases: for $P=\Proj^V_W$, the $e_i$'s and $f_i$'s are orthonormal eigenvectors of $P^*P$ and $PP^*$, respectively, and the $\cos\theta_i$'s are square roots of the eigenvalues of $P^*P$, if $p\leq q$, or $PP^*$ otherwise.
The $\theta_i$'s are uniquely defined, but the $e_i$'s and $f_i$'s are not.
$P$ is given in principal bases by a $q\times p$ diagonal matrix formed with the $\cos\theta_i$'s, as
\begin{equation}\label{eq:Pei}
	Pe_i=\begin{cases}
		f_i\cdot\cos\theta_i \ \text{ if } 1\leq i\leq m, \\
		0 \hspace{37pt}\text{ if } i>m.
	\end{cases}
\end{equation}

A geometric interpretation of principal angles is that the unit sphere of $V$ projects to an ellipsoid in $W$ with semi-axes of lengths $\cos\theta_i$, for $1\leq i\leq m$ (in the complex case there are 2 semi-axes for each $i$).
They can also be described recursively with a minimization condition: $e_1$ and $f_1$ form the smallest angle $\theta_1$ between nonzero vectors of $V$ and $W$; in their orthogonal complements we obtain $e_2$, $f_2$ and $\theta_2$ in the same way; and so on.

\subsection{Grassmann angles and blade products}\label{sc:Grassmann angle}

Grassmann angles were introduced in \cite{Mandolesi_Grassmann}, and related to blade products in \cite{Mandolesi_Products}.

\begin{definition}
Let $V,W\subset X$ be subspaces, $\nu$ be a nonzero blade representing $V$, and $P=\Proj_W$.
The \emph{Grassmann angle} $\Theta_{V,W}\in[0,\frac\pi 2]$ is given by
\SELF{$V=0$ : $P\nu=\nu$, $\Theta=0$. \\ 
	  $V\neq 0$, $W=0$ : $\nu\neq 0$, $P\nu=0$, $\Theta=\frac\pi 2$
}
\begin{equation}\label{eq:norm projection blade}
\cos\Theta_{V,W}=\frac{\|P\nu\|}{\|\nu\|}.
\end{equation}
\end{definition}

As blade norms (squared, in the complex case) give volumes, $\Theta_{V,W}$ tells us how volumes in $V$ contract when orthogonally projected on $W$. 
In simple cases, where there is an unambiguous concept of angle between subspaces, $\Theta_{V,W}$ coincides with it, as when $V$ is a line, or $V$ and $W$ are planes in $\R^3$.

The main difference between Grassmann angles and similar ones \cite{Gluck1967,Gunawan2005,Hitzer2010a,Jiang1996} is its asymmetry: $\Theta_{V,W}=\frac\pi 2$ if $\dim V>\dim W$, so, in general, $\Theta_{V,W}\neq \Theta_{W,V}$ when dimensions are different.
It reflects the dimensional asymmetry between the subspaces, leading to better and more general results. 
For example, the contraction and exterior product formulas in \cref{pr:products}, or \cref{pr:formula any base dimension,pr:formula complementary angle bases}, only hold without any restrictions thanks to this asymmetry.

Grassmann angles work well in complex spaces, and this makes them useful for applications in areas like quantum theory. Differences due to volumes being given, in this case, by squared norms, have important implications \cite{Mandolesi_Pythagorean,Mandolesi_Born}.

These angles have many useful properties, some of which are listed below.

\begin{proposition}\label{pr:properties Grassmann}
	Let $V,W\subset X$ be subspaces, with principal angles $\theta_1,\ldots,\theta_m$, where $m=\min\{p,q\}$ for $p=\dim V$ and $q=\dim W$, and let $P=\Proj^V_W$. 
	\begin{enumerate}[i)]
	\item $\Theta_{V,W}=\frac \pi 2 \ \Leftrightarrow\ V \pperp W$. \label{it:Theta pi2}
	\item If $V=\Span(v)$ and $W=\Span(w)$ for nonzero $v,w\in X$ then $\Theta_{V,W} = \min\{\theta_{v,w},\pi-\theta_{v,w}\}$ in the real case, and  $\Theta_{V,W} = \gamma_{v,w}$ in the complex one. \label{it:Theta lines}
	\item If $V$ is a line and $v\in V$ then $\|P v\|=\|v\|\cos\Theta_{V,W}$. \label{it:Pv}
	\item $\vol P(S)=\vol S \cdot\cos\Theta_{V,W}$ ($\cos^2\Theta_{V,W}$ in complex case), where $S\subset V$ is a parallelotope and  $\vol$ is the $p$-dimensional volume ($2p$ in complex case). \label{it:projection factor}
	\item $\cos^2\Theta_{V,W}=\det(\bar{\mathbf{P}}^T \mathbf{P})$, where $\mathbf{P}$ is a matrix for $P$ in orthonormal bases of $V$ and $W$. \label{it:formula orthonormal bases}
	\item If $p> q$ then $\Theta_{V,W}=\frac\pi 2$, otherwise $\cos\Theta_{V,W}=\prod_{i=1}^m \cos\theta_i$. 
	\item If $p=q$ then $\Theta_{V,W}= \Theta_{W,V}$.
	\item $\Theta_{T(V),T(W)} = \Theta_{V,W}$ for any orthogonal (unitary, in the complex case) transformation $T:X\rightarrow X$.\label{it:transformation}
		\item If $V'$ and $W'$ are the orthogonal complements of $V\cap W$ in $V$ and $W$, respectively, then $\Theta_{V,W}=\Theta_{V',W'}$. \label{it:orth complem inter}
	\item $\Theta_{V,W} = \Theta_{W^\perp,V^\perp}$. \label{it:Theta perp perp}
\end{enumerate}
\end{proposition}

The Grassmann angle with an orthogonal complement has extra properties which grant it a special name and notation.

\begin{definition}
	The \emph{complementary Grassmann angle} $\Theta_{V,W}^\perp \in [0,\frac \pi 2]$ of subspaces $V,W\subset X$ is $\Theta_{V,W}^\perp=\Theta_{V,W^\perp}$.
\end{definition}

In general, this is not the usual complement, i.e. $\Theta_{V,W}^\perp\neq\frac \pi 2-\Theta_{V,W}$.

\begin{proposition}\label{pr:complementary simple cases}
	Let $V,W\subset X$ be subspaces.
	\begin{enumerate}[i)]
		\item $\Theta_{V,W}^\perp=0 \ \Leftrightarrow\ $ $V\perp W$.\label{it:Theta perp 0}
		\item $\Theta_{V,W}^\perp=\frac \pi 2 \ \Leftrightarrow\ V\cap W\neq\{0\}$. \label{it:Theta perp pi2}
		\item If $\theta_1,\ldots,\theta_m$ are the principal angles then $\cos\Theta_{V,W}^\perp=\prod_{i=1}^m \sin\theta_i$. \label{it:complementary product sines}
		\item If $V$ is a line then $\Theta_{V,W}^\perp=\frac \pi 2-\Theta_{V,W}$. \label{it:complementary line}
		\item $\Theta_{V,W}^\perp = \Theta_{W,V}^\perp$. \label{it:symmetry complementary}
	\end{enumerate}
\end{proposition}

In \cref{sc:Formulas} we get a simpler way to prove \emph{(iii)} than the one in \cite{Mandolesi_Grassmann}.

\begin{definition}
	The \emph{(left) contraction} of $\nu\in\Lambda^p X$ on $\omega\in\Lambda^q X$ is the unique $\nu\lcontr\omega\in \Lambda^{q-p} X$ such that, for all $\mu\in\Lambda^{q-p} X$,
	\begin{equation}\label{eq:contraction}
		\inner{\mu,\nu \lcontr  \omega} = \inner{\nu\wedge\mu,\omega}.
	\end{equation}
\end{definition}

This contraction coincides with the inner product when $p=q$, and it is asymmetric, with $\nu\lcontr\omega=0$ if $p>q$.
It differs from the one used in Clifford geometric algebra \cite{Dorst2002} by a reversion.

\begin{proposition}
	For $\nu\in\Lambda^p X$ and any blade $\omega\in\Lambda^q X$, with $p\leq q$, 
	\SELF{Inclui $p=0$}
	\SELF{Geometric Algebra, Chisolm, eq. 95, dá uma fórmula semelhante pra $\glcontr$}
	\begin{equation}\label{eq:contraction coordinate decomposition} 
		\nu \lcontr  \omega = \sum_{I\in\I_p^q} \sigma_I \,\inner{\nu,\omega_I}\, \omega_{\hat{I}},
	\end{equation}
	where $\sigma_I$, $\omega_I$ and $\omega_{\hat{I}}$ are as in \eqref{eq:multiindex decomposition} for any decomposition $\omega = w_1\wedge\ldots\wedge w_q$.
\end{proposition}

\begin{proposition}\label{pr:products}
Let $\nu,\omega\in\Lambda X$ be blades representing $V,W\subset X$, respectively.
\begin{enumerate}[i)]
\item $|\inner{\nu,\omega}| = \|\nu\|\|\omega\|\cos \Theta_{V,W}$, if $\nu$ and $\omega$ have equal grades. \label{it:Theta inner blades}
\item $\|\nu\lcontr\omega\| = \|\nu\|\|\omega\| \cos \Theta_{V,W}$. \label{it:Theta norm contraction}
\item $\|\nu\wedge\omega\|=\|\nu\|\|\omega\| \cos\Theta^\perp_{V,W}$. \label{it:exterior product} 
\end{enumerate}
\end{proposition}

With a Grassmann angle $\mathbf{\Theta}_{V,W}$ for oriented subspaces (with orientations of $\nu,\omega\in\Lambda^p X$) given by $\cos \mathbf{\Theta}_{V,W} = \frac{\inner{\nu,\omega}}{|\inner{\nu,\omega}|} \cos\Theta_{V,W}$ if $\inner{\nu,\omega}\neq 0$, otherwise $\mathbf{\Theta}_{V,W} = \frac\pi 2$, we also have
\begin{equation}\label{eq:inner oriented}
	\inner{\nu,\omega} = \|\nu\|\|\omega\|\cos \mathbf{\Theta}_{V,W}.
\end{equation}

\section{Formulas for Grassmann angles}\label{sc:Formulas}

Here we give formulas for computing Grassmann and complementary Grassmann angles in terms of arbitrary bases, thus generalizing \cref{pr:properties Grassmann}\emph{\ref{it:formula orthonormal bases}}.
They can be adapted for use with similar angles \cite{Gluck1967,Gunawan2005,Hitzer2010a,Jiang1996}, with some dimensional restrictions ($\dim V\leq \dim W$ in \cref{pr:formula any base dimension}, $\dim V\leq \dim W^\perp$ in \cref{pr:formula complementary angle bases}) as these other angles are symmetric.

The first formula, for subspaces of same dimension, follows from \cref{pr:products}\emph{\ref{it:Theta inner blades}}. A similar result, for the real case, appears in \cite{Gluck1967}.

\begin{theorem}\label{pr:formula Theta equal dim}
Given bases $(v_1,\ldots,v_p)$ and $(w_1,\ldots,w_p)$ of $p$-dimensional subspaces $V,W\subset X$, let $A=\big(\inner{w_i,w_j}\big), B=\big(\inner{w_i,v_j}\big)$ and $D=\big(\inner{v_i,v_j}\big)$. Then
\begin{equation*}
\cos^2 \Theta_{V,W} = \frac{|\det B\,|^2}{\det A \cdot\det D}.
\end{equation*}
\end{theorem}

\begin{example}\label{ex:formula bases}
In $\C^3$, let  $V=\Span(v_1,v_2)$ and $W=\Span(w_1,w_2)$ with $v_1=(1,-\xi,0), v_2=(0,\xi,-\xi^2)$, $w_1=(1,0,0)$ and $w_2=(0,\xi,0)$,  where $\xi=e^{\im\frac{2\pi}{3}}$. The theorem gives 
$\Theta_{V,W}=\arccos\frac{\sqrt{3}}{3}$. 
Since $v=(\xi,\xi^2,-2)\in V$ and $w=(1,\xi,0)\in W$ are orthogonal to $V\cap W=\Span(v_1)$, \cref{pr:properties Grassmann}\emph{\ref{it:Theta lines}} and \emph{\ref{it:orth complem inter}} give $\Theta_{V,W}=\Theta_{\Span(v),\Span(w)}=\gamma_{v,w}$, and the Hermitian angle formula confirms the result.
\SELF{Usa \cref{pr:properties Grassmann}\ref{it:Theta lines}}
\end{example}

The next formulas require some determinant identities.

\begin{proposition}[Laplace's Expansion \protect{\cite[p.80]{Muir2003}}]\label{pr:Laplace expansion}
	Given a $q\times q$ matrix $M$ and  a multi-index $J\in\I_p^q$, with $1\leq p<q$,  
	\begin{equation*}
		\det M = \sum_{I\in\I_p^q} (-1)^{|I|+|J|} \det M_{I,J} \cdot \det M_{\hat{I},\hat{J}},
	\end{equation*}
	where $M_{I,J}$ is the $p\times p$ submatrix formed by entries with row indices in $I$ and column indices in $J$, and $M_{\hat{I},\hat{J}}$ is the $(q-p)\times(q-p)$ submatrix formed by entries with row indices not in $I$ and column indices not in $J$.
\end{proposition}


\begin{proposition}[Schur's determinant identity \cite{Brualdi1983}] \label{pr:Schur}
Let $M=\begin{psmallmatrix}
A & B \\ 
C & D
\end{psmallmatrix}$ be a $(q+p)\times(q+p)$ matrix, partitioned into $q\times q$, $q\times p$, $p\times q$ and $p\times p$ matrices $A$, $B$, $C$ and $D$, respectively. If $A$ is invertible then
\begin{equation}\label{eq:Schur A}
\det M = \det A \cdot\det(D-CA^{-1}B).
\end{equation}
Likewise, if $D$ is invertible then 
\begin{equation}\label{eq:Schur D}
\det M = \det D\cdot\det(A-BD^{-1}C).
\end{equation}
\end{proposition}
\begin{proof}
Follows by decomposing $M$ into block triangular matrices, as
$M=\begin{psmallmatrix}
A & 0_{q\times p} \\ 
C & \mathds{1}_{p\times p}
\end{psmallmatrix} 
\begin{psmallmatrix}
\mathds{1}_{q\times q} & A^{-1}B \\ 
0_{p\times q} & D-CA^{-1}B
\end{psmallmatrix}$
or 
$M=\begin{psmallmatrix}
	A-BD^{-1}C & BD^{-1} \\ 
	0_{p\times q} & \mathds{1}_{p\times p}
\end{psmallmatrix} 
\begin{psmallmatrix}
	\mathds{1}_{q\times q} & 0_{q\times p} \\ 
	C & D
\end{psmallmatrix}$.
\SELF{det of block triangular matrices = product of det of diagonal blocks (by \cref{pr:Laplace expansion})}
\end{proof}

We now get a formula for distinct dimensions, simpler than one given in \cite{Gunawan2005} for a similar angle (which corrects another formula from \cite{Risteski2001}).

\begin{theorem}\label{pr:formula any base dimension}
Given bases $(v_1,\ldots,v_p)$ of $V$ and $(w_1,\ldots,w_q)$ of $W$, let $A=\big(\inner{w_i,w_j}\big), B=\big(\inner{w_i,v_j}\big)$, and $D=\big(\inner{v_i,v_j}\big)$. Then
\begin{equation*} 
\cos^2 \Theta_{V,W} = \frac{\det(\bar{B}^T \! A^{-1}B)}{\det D}.
\end{equation*}
\end{theorem}
\begin{proof}
If $p>q$ then $\Theta_{V,W} = \frac \pi 2$, and the determinant of $\bar{B}^T \! A^{-1}B$ vanishes as it is a $p\times p$ matrix with rank at most $q$. 

If $p\leq q$, applying \cref{pr:Laplace expansion}, with $J=(q+1,\ldots,q+p)$, to the $(q+p)\times(q+p)$ block matrix $M=\begin{psmallmatrix}
A & B\  \\ 
\bar{B}^T & 0_{p\times p}
\end{psmallmatrix}$, we get
\SELF{Como $J$ dá as últimas colunas, das $q+p$ linhas de $M$ só as $q$ primeiras não dão 0, por isso pode usar $\I_p^q$ ao invés de $\I_p^{q+p}$}
\[ \det M = \sum_{I\in\I_p^q} (-1)^{|I|+pq+\frac{p(p+1)}{2}}\cdot \det B_I \cdot\det N_{\hat{I}}, \]
where $B_I$ is the $p\times p$ submatrix of $M$ formed by lines of $B$ with indices in $I$, and $N_{\hat{I}}=\begin{psmallmatrix}
A_{\hat{I}} \\[1pt] \bar{B}^T
\end{psmallmatrix}$ is its $q\times q$ complementary submatrix, formed by lines of $A$ with indices not in $I$ and all of $\bar{B}^T$.

For $\nu=v_1\wedge\ldots \wedge v_p$ and $\omega=w_1\wedge\ldots\wedge w_q$ we have, by \eqref{eq:contraction} and \eqref{eq:contraction coordinate decomposition},
\begin{equation*}
\|\nu\lcontr\omega\|^2 = \inner{\nu\wedge(\nu\lcontr\omega),\omega} = \sum_{I\in\I_p^q} \sigma_I \,\inner{\omega_I,\nu}\, \inner{\nu\wedge\omega_{\hat{I}},\omega}.
\end{equation*}
Since $\det B_I=\inner{\omega_I,\nu}$, $\det N_{\hat{I}} = \inner{\omega_{\hat{I}}\wedge\nu,\omega} = (-1)^{pq+p}\inner{\nu\wedge\omega_{\hat{I}},\omega}$
\SELF{$(-1)^{-p^2}=(-1)^p$}
and $\sigma_I = (-1)^{|I|+\frac{p(p+1)}{2}}$, we obtain $\|\nu\lcontr\omega\|^2 = (-1)^p \det M$.
\Cref{pr:products}\emph{\ref{it:Theta norm contraction}} then gives $\cos^2 \Theta_{V,W} = \frac{(-1)^p \det M}{\det D\det A}$, and the result follows from \eqref{eq:Schur A}.\SELF{$\det(-BA^{-1}\bar{B}^T) = (-1)^p \det(BA^{-1}\bar{B}^T)$}
\end{proof}

\begin{example}\label{ex:formula distinct dim}
	In $\R^4$, let $v=(1,0,1,0)$, $w_1=(0,1,1,0)$, $w_2=(1,2,2,-1)$, $V=\Span(v)$ and $W=\Span(w_1,w_2)$. 
	Then $A=\begin{psmallmatrix}
		2 & 4 \\ 4 & 10
	\end{psmallmatrix}$,
	$B=\begin{psmallmatrix}
		1 \\ 3
	\end{psmallmatrix}$ and $D=(2)$, and the theorem gives $\Theta_{V,W}= 45^\circ$, as one can verify by projecting $v$ on $W$. 
	Switching the roles of $V$ and $W$, we now have $A=(2)$,
	$B=(1 \ 3)$, $D=\begin{psmallmatrix}
	2 & 4 \\ 4 & 10
	\end{psmallmatrix}$ and $\Theta_{W,V}= 90^\circ$, which is correct since $\dim W>\dim V$.
\end{example}

We now get formulas for the complementary Grassmann angle.

\begin{theorem}\label{pr:formula complementary angle bases}
Given bases $(v_1,\ldots,v_p)$ of $V$ and $(w_1,\ldots,w_q)$ of $W$, let $A=\big(\inner{w_i,w_j}\big), B=\big(\inner{w_i,v_j}\big)$, and $D=\big(\inner{v_i,v_j}\big)$. Then
\begin{equation}\label{eq:complementary bases}
\cos^2 \Theta^\perp_{V,W} = \frac{\det(A-BD^{-1}\bar{B}^T )}{\det A}.
\end{equation}
\end{theorem}
\begin{proof}
Let $\nu=v_1\wedge\ldots\wedge v_p$ and $\omega=w_1\wedge\ldots\wedge w_q$.
The result is obtained applying \cref{pr:products}\emph{\ref{it:exterior product}} to $\omega\wedge\nu$, 
and using \eqref{eq:Schur D} with 
$M=\begin{psmallmatrix}
A & B \\ 
\bar{B}^T & D
\end{psmallmatrix}$.
\end{proof}


\begin{corollary}
	If $\mathbf{P}$ is a matrix representing $\Proj^V_W$ in orthonormal bases of $V$ and $W$ then
	\begin{equation}\label{eq:complementary orthonormal bases}
	\cos^2\Theta_{V,W}^\perp=\det(\mathds{1}_{q\times q}-\mathbf{P}\bar{\mathbf{P}}^T ).
	\end{equation}
\end{corollary}

This gives an easy way to prove \cref{pr:complementary simple cases}\emph{\ref{it:complementary product sines}}, as in principal bases $\mathbf{P}$ is a diagonal matrix formed with the $\cos\theta_i$'s.

\begin{example}
	In \cref{ex:formula distinct dim}, \eqref{eq:complementary bases} gives $\Theta^\perp_{V,W}=45^\circ$, in agreement with \cref{pr:complementary simple cases}\emph{\ref{it:complementary line}}. The same formula also gives $\Theta^\perp_{W,V}=45^\circ$, as expected by \cref{pr:complementary simple cases}\emph{\ref{it:symmetry complementary}}.
	Direct calculations show the principal angles of $V$ and $W^\perp$ are $0^\circ$ and $45^\circ$, as are those of $W$ and $V^\perp$, confirming the results.
\end{example}

\begin{example}
	In \cref{ex:formula bases}, using \eqref{eq:complementary bases} with the bases $(v_1,v_2)$ and $(w_1,w_2)$, or  \eqref{eq:complementary orthonormal bases} with the orthonormal bases $(\frac{v_1}{\sqrt{2}},\frac{v}{\sqrt{6}})$ and $(\frac{v_1}{\sqrt{2}},\frac{w}{\sqrt{2}})$, we get $\Theta^\perp_{V,W}=90^\circ$, as expected by \cref{pr:complementary simple cases}\emph{\ref{it:Theta perp pi2}}, since $V\cap W\neq\{0\}$.
\end{example}

\section{Generalized Pythagorean identities}\label{sc:Pythagorean trigonometric}

The Pythagorean trigonometric identity $\cos^2\theta+\sin^2\theta=1$ can be written as 
$\cos^2\theta_x+\cos^2\theta_y=1$, with $\theta_x$ and $\theta_y$ being angles a line in $\R^2$ makes with the axes. 
We give generalizations for Grassmann angles which, with \cref{pr:properties Grassmann}\emph{\ref{it:projection factor}}, lead to real and complex Pythagorean theorems for volumes \cite{Mandolesi_Pythagorean}.
Some correspond to known results in real spaces, which are now extended to the complex case, with important implications for quantum theory \cite{Mandolesi_Born}. We also get a geometric interpretation for a property of the Clifford product.

The first identity relates the Grassmann angles of a (real or complex) line with all subspaces of an orthogonal partition  of $X$.

\begin{theorem}
Given an orthogonal partition $X=W_1\oplus\cdots\oplus W_k$ and a line $L\subset X$, 
\begin{equation}\label{eq:pythagorean line}
	\sum_{i=1}^k \cos^2\Theta_{L,W_i} = 1.
\end{equation}
\end{theorem}
\begin{proof}
Given a nonzero $v\in L$, as $\|v\|^2 = \sum_{i} \left\|\Proj_{W_i} v \right\|^2$ the result follows from \cref{pr:properties Grassmann}\emph{\ref{it:Pv}}.
\end{proof}

\begin{figure}
	\centering
	\subfloat[$\cos^2\theta_x+\cos^2\theta_y+\cos^2\theta_z=1$]{
		\resizebox*{0.38\textwidth}{!}{\includegraphics{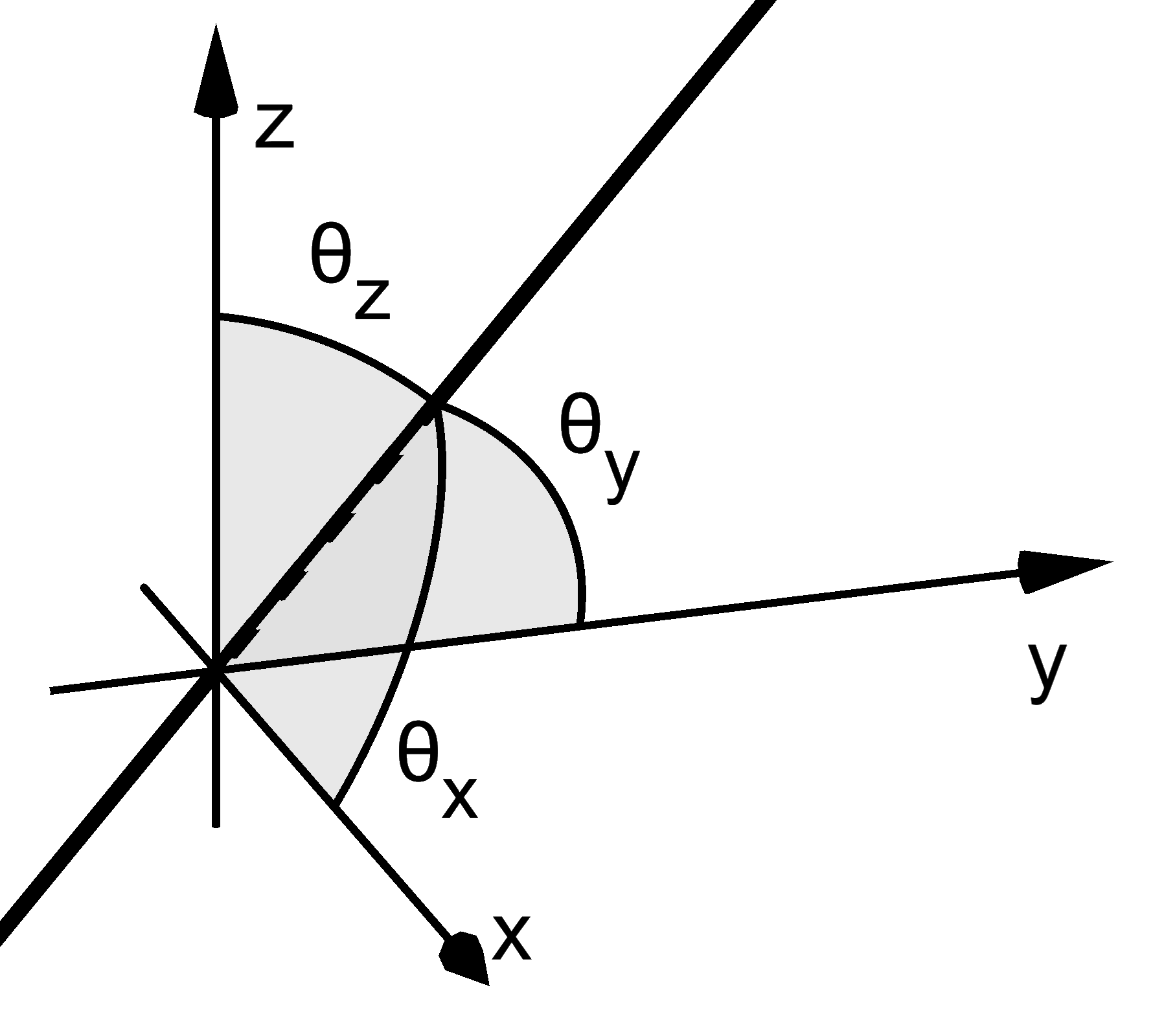}}\label{fig:angulos-eixos-edit}}
	\hfill
	\subfloat[$\cos^2\theta_{xy}+\cos^2\theta_{xz}+\cos^2\theta_{yz}=1$]{
		\resizebox*{0.41\textwidth}{!}{\includegraphics{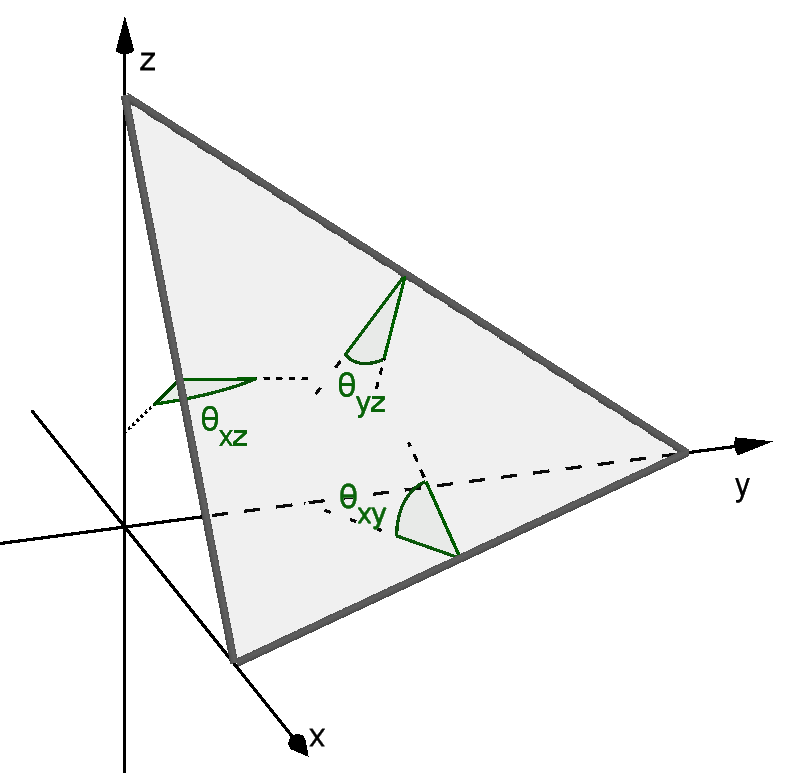}}\label{fig:angulos-faces-edit}}
	\caption{Pythagorean identities for subspaces of equal dimensions.}
	\label{fig:equal dimensions}
\end{figure}

\begin{example}\label{ex:direction cosines}
If $\theta_x, \theta_y$ and $\theta_z$ are the angles between a line in $\R^3$ and the axes (\cref{fig:angulos-eixos-edit}), then $\cos^2\theta_x+\cos^2\theta_y+\cos^2\theta_z=1$. 
\end{example}

This is a known identity for direction cosines and, like other examples we give, is only meant to illustrate the theorem in $\R^3$. The relevance of our result lies mainly in the complex case, where it has important connections with quantum theory.

\begin{example}
	If $X$ is the complex Hilbert space of a quantum system, $L$ is the complex line of a quantum state vector $\psi$, and the $W_i$'s are the eigenspaces of a quantum observable,
	the probability of getting result $i$ when measuring $\psi$ is  $p_i=\|\Proj_{W_i}\psi\|^2/\|\psi\|^2$ \cite{CohenTannoudji2019}. So, by \cref{pr:properties Grassmann}\emph{\ref{it:Pv}}, $p_i=\cos^2 \Theta_{L,W_i}$, and \eqref{eq:pythagorean line} reflects the fact that the total probability is 1.
\end{example}

By \cref{pr:properties Grassmann}\emph{\ref{it:projection factor}}, $\cos^2\Theta_{L,W_i}$ measures area contraction in orthogonal projections of the complex line $L$. This is explored in \cite{Mandolesi_Born} to get a new interpretation for quantum probabilities and derive the Born rule.

The next identities relate Grassmann angles of a (real or complex) subspace with coordinate subspaces of an orthogonal basis of $X$.

\begin{theorem}\label{pr:Grassmann coordinate}
If $V\subset X$ is a $p$-dimensional subspace, and $n=\dim X$,
\[ \sum_{I\in\I_p^n} \cos^2\Theta_{V,W_I} = 1, \]
where the $W_I$'s are the coordinate $p$-subspaces of an orthogonal basis of $X$.
\end{theorem}
\begin{proof}
Without loss of generality, assume the basis is orthonormal, so its coordinate $p$-blades $\omega_I$ form an orthonormal basis of $\Lambda^p X$.
For an unit blade $\nu\in\Lambda^p V\subset \Lambda^p X$ we have $\sum_I |\inner{\nu,\omega_I}|^2=1$, and the result follows from \cref{pr:products}\emph{\ref{it:Theta inner blades}}.
\end{proof}

A similar result for the real case, in terms of products of cosines of principal angles, appears in \cite{Miao1992}. 
The theorem extends to affine subspaces, as in the next example, which is a dual of \cref{ex:direction cosines} via \cref{pr:properties Grassmann}\emph{\ref{it:Theta perp perp}}.

\begin{example}
If $\theta_{xy}, \theta_{xz}$ and $\theta_{yz}$ are the angles a plane in $\R^3$ makes with the coordinate planes (\cref{fig:angulos-faces-edit}) then $\cos^2\theta_{xy}+\cos^2\theta_{xz}+\cos^2\theta_{yz}=1$.
\SELF{Another way to express this result is that the sum of the squared cosines of all angles 
	  between the faces of a trirectangular tetrahedron equals 1. By \cref{pr:Grassmann coordinate}, this generalizes for simplices of any dimension {Cho1992}.
}
\end{example}

\begin{example}
Let $\xi,w_1,w_2$ and $V$ be as in \cref{ex:formula bases}, $w_3=(0,0,\xi^2)$ and $W_{ij}=\Span(w_i,w_j)$. 
As the unitary transformation given by $T=\left(\begin{smallmatrix}
	0 & 0 & \xi \\ 
	\xi & 0 & 0 \\ 
	0 & \xi & 0
\end{smallmatrix}\right)$ maps $W_{12}\mapsto W_{23}$, $W_{23}\mapsto W_{13}$, and preserves $V$, \cref{pr:properties Grassmann}\emph{\ref{it:transformation}} gives $\Theta_{V,W_{12}} = \Theta_{V,W_{23}} = \Theta_{V,W_{13}}$.
\SELF{$T:w_1\mapsto w_2\mapsto w_3\mapsto w_1$, $Tv_1=v_2$,\\ $Tv_2= -(v_1+v_2)$}
Since $W_{12}$, $W_{13}$ and $W_{23}$ are the coordinate $2$-subspaces of the orthogonal basis $(w_1,w_2,w_3)$  of $\C^3$, \cref{pr:Grassmann coordinate} gives $\cos\Theta_{V,W_{ij}}=\frac{\sqrt{3}}{3}$, in agreement with that example.
\end{example}

A formula relating the geometric product of Clifford algebra \cite{Hestenes1984clifford} to Grassmann angles, given in \cite{Mandolesi_Products}, implies $\|AB\|^2 = \|A\|^2\|B\|^2 \sum_{J} \cos^2\Theta_{V,Y_J}$, where $A$ and $B$ are $p$-blades representing subspaces $V$ and $W$, and the $Y_J$'s are all coordinate $p$-subspaces of a certain orthogonal basis of $Y=V+W$.
This gives a geometric interpretation for a simple yet important property of this product. 
Those in \cref{pr:products} are submultiplicative for blades because they correspond to projections on single subspaces. The geometric product, on the other hand, involves projections on all $Y_J$'s, and with \cref{pr:Grassmann coordinate} we see this is what allows $\|AB\|= \|A\|\|B\|$.

Extending a result of \cite{Miao1996}, we also have identities for Grassmann angles with coordinate subspaces of a  dimension different from $V$.

\begin{theorem}
Let $V\subset X$ be a $p$-dimensional subspace, $0\leq q\leq n=\dim X$, and the $W_I$'s be the coordinate $q$-subspaces of an orthogonal basis of $X$.
\begin{enumerate}[i)]
\item If $p\leq q$ then $\displaystyle \sum_{I\in\I_q^n} \cos^2\Theta_{V,W_I} =\binom{n-p}{n-q}$. 
\item If $p>q$ then $\displaystyle \sum_{I\in\I_q^n} \cos^2\Theta_{W_I,V} = \binom{p}{q}$. 
\end{enumerate}
\end{theorem}
\begin{proof}
We can assume $p,q\neq 0$ and that the basis is orthonormal. 
So, for $0\leq r\leq n$ and with $\omega_I$'s as in \eqref{eq:coordinate blades}, $\{\omega_I\}_{I\in\I_r^n}$ and $\{\omega_{\hat{I}}\}_{I\in\I_r^n}$ are orthonormal bases of $\Lambda^r X$ and $\Lambda^{n-r} X$, respectively. 
\begin{itemize}
\item[\emph{(i)}] For an unit blade $\nu\in\Lambda^p V$ and  $I=(i_1,\ldots,i_q)\in\I_q^n$ we have, by \eqref{eq:contraction coordinate decomposition},
\[ \nu \lcontr  \omega_I = \sum_{J\in\I_p^q} \sigma_J \,\inner{\nu,(\omega_I)_J}\, (\omega_I)_{\hat{J}}, \]
where $(\omega_I)_J = w_{ i_{j_1}}\wedge\ldots\wedge w_{ i_{j_p}}$ for $J=(j_1,\ldots,j_p)$, and likewise for $(\omega_I)_{\hat{J}}$.
As the $(\omega_I)_{\hat{J}}$'s are orthonormal, \cref{pr:products}\emph{\ref{it:Theta norm contraction}} gives
\[ \cos^2\Theta_{V,W_I} = \|\nu \lcontr  \omega_I\|^2 = \sum_{J\in\I_p^q} \left|\inner{\nu,(\omega_I)_J}\right|^2, \]
and therefore
\begin{align*}
\sum_{I\in\I_q^n} \cos^2\Theta_{V,W_I} 
&= \sum_{I\in\I_q^n} \sum_{J\in\I_p^q} \left|\inner{\nu,(\omega_I)_J}\right|^2 \\
&= \frac{\binom{n}{q} \binom{q}{p}}{\binom{n}{p}} \sum_{K\in\I_p^n} \left|\inner{\nu,\omega_K}\right|^2 = \binom{n-p}{n-q} \|\nu\|^2,
\end{align*}
where the binomial coefficients account for the number of times each $\omega_K$ appears as a $(\omega_I)_J$ in the double summation.

\item[\emph{(ii)}] For each $I\in\I_q^n$, \cref{pr:properties Grassmann}\emph{\ref{it:Theta perp perp}} gives $\Theta_{W_I,V} = \Theta_{V^\perp,{W_I}^\perp}$. 
As ${W_I}^\perp = W_{\hat{I}}$ for $\hat{I}\in\I_{n-q}^n$, and $\dim V^\perp = n-p < n-q = \dim {W_I}^\perp$, the result follows from the previous case. \qedhere
\end{itemize}
\end{proof}

The following examples are again duals of each other. 

\begin{figure}
	\centering
	\subfloat[$\cos^2\theta_{xy}+\cos^2\theta_{xz}+\cos^2\theta_{yz}=2$]{
		\resizebox*{0.37\textwidth}{!}{\includegraphics{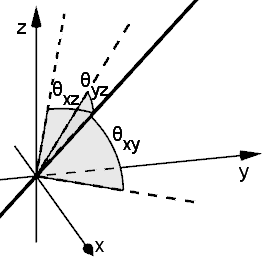}}\label{fig:angulos-linha-planos}}
	\hfill
	\subfloat[$\cos^2\theta_{x}+\cos^2\theta_{y}+\cos^2\theta_{z}=2$]{
		\resizebox*{0.42\textwidth}{!}{\includegraphics{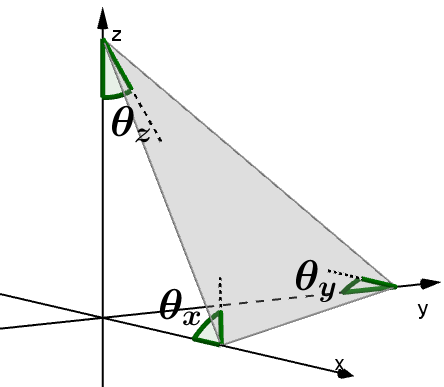}}\label{fig:angulos plano eixos}}
	\caption{Pythagorean identities for subspaces of different dimensions.}
	\label{fig:diferent dimensions}
\end{figure}

\begin{example}
If $\theta_{xy}, \theta_{xz}$ and $\theta_{yz}$ are the angles a line in $\R^3$ makes with the coordinate planes (\cref{fig:angulos-linha-planos}) then   $\cos^2\theta_{xy}+\cos^2\theta_{xz}+\cos^2\theta_{yz}=2$.
\end{example}

\begin{example}
If $\theta_{x}, \theta_{y}$ and $\theta_{z}$ are the angles between the axes and a plane in $\R^3$ (\cref{fig:angulos plano eixos}) then $\cos^2\theta_{x}+\cos^2\theta_{y}+\cos^2\theta_{z}=2$.
\end{example}

\section{Other identities}\label{sc:other identities}

Using the Grassmann angle $\mathbf{\Theta}_{V,W}$ for oriented subspaces, we have:

\begin{theorem}
	Given $p$-dimensional oriented subspaces $V,W\subset X$,
	\begin{equation*}
		\cos\mathbf{\Theta}_{V,W} = \sum_{I\in\I_p^n} \cos\mathbf{\Theta}_{V,X_I} \cos\mathbf{\Theta}_{W,X_I},
	\end{equation*}
	where the $X_I$'s are the coordinate $p$-subspaces of an orthogonal basis of $X$, with orientations given by the corresponding coordinate $p$-blades.
\end{theorem}
\begin{proof}
	The result follows by decomposing unit blades $\nu\in\Lambda^p V$ and $\omega\in\Lambda^p W$ (with the orientations of $V$ and $W$) in the orthonormal basis of $\Lambda^p X$ formed with the normalized coordinate $p$-blades, and applying \eqref{eq:inner oriented}.
\end{proof}

This gives an inequality for $\cos\Theta_{V,W} = |\cos\mathbf{\Theta}_{V,W}|$, like one from \cite{Miao1996}.

\begin{corollary}
$\cos\Theta_{V,W} \leq \sum_I \cos\Theta_{V,X_I} \cos\Theta_{W,X_I}$.
\end{corollary}

\begin{example}
	$\cos\theta = \cos\alpha_x\cos\beta_x + \cos\alpha_y\cos\beta_y + \cos\alpha_z\cos\beta_z$ for the angle $\theta\in[0,\pi]$ between 2 oriented lines in $\R^3$ forming angles $\alpha_x,\alpha_y,\alpha_z$ and $\beta_x,\beta_y,\beta_z$ (all in $[0,\pi]$) with the positive axes.
\end{example}

The last identities will require some preparation. 

\begin{definition}
	A coordinate subspace of a principal basis $\beta$ of $V$ w.r.t. $W$ is a \emph{principal subspace}\footnote{Some authors use `principal subspace' for $\Span(e_i,f_i)$, where $e_i\in V$ and $f_i\in W$ are principal vectors corresponding to the same principal angle $\theta_i$.} (of $V$ w.r.t. $W$, for $\beta$). 
	Two or more subspaces of $V$ are \emph{coprincipal} (w.r.t. $W$) if they are principal for the same $\beta$.
\end{definition}

Note that  a subspace being principal depends on both $V$ and $W$, even if they are left implicit. 
Also, $\{0\}$ is always principal.

\begin{lemma}\label{pr:orthogonal subspace principal}
Let $V,W\subset X$ be nonzero subspaces and $U\subset V$ be any subspace. If $U\perp W$ then $U$ is principal w.r.t. $W$.
\SELF{Usa p/\cref{pr:principal orth proj}. Se $U\subset V\cap W$ também vale} 
\end{lemma}
\begin{proof}
By \eqref{eq:inner ei fj}, the union of an orthonormal basis of $U$ and a principal basis of $U^\perp\cap V$ w.r.t. $W$ gives a principal basis of $V$ w.r.t. $W$. 
\end{proof}

\begin{lemma}\label{pr:principal subspaces}
Let $V,W\subset X$ be nonzero subspaces, with associated principal bases $\beta_V$ and $\beta_W$, $U\subset V$ be any subspace, and $P=\Proj_W$. Then:
\begin{enumerate}[i)]
\item $U$ is principal for $\beta_V$ $\Leftrightarrow$ $U^\perp \cap V$ is principal for $\beta_V$.\label{it:complement principal}
\SELF{Usa para \cref{pr:principal orth proj}}
\item $U$ is principal for $\beta_V$ $\Rightarrow$ $P(U)$ is principal for $\beta_W$. The converse holds if $V\not\pperp W$.\label{it:P(U) principal}
\SELF{\cref{pr:principal partition} usa ida}
\end{enumerate}
\end{lemma}
\begin{proof}
	\emph{(i)} If $U$ is spanned by some vectors of the orthogonal basis $\beta_V$, $U^\perp \cap V$ is spanned by the others.
	\emph{(ii)} Follows from \eqref{eq:Pei}, and the converse from \cref{pr:partial orthogonality}
\SELF{\emph{\ref{it:orthogonality}} and \emph{\ref{pr:subspace not pperp}}}
	as well. 
\end{proof}

\begin{proposition}\label{pr:principal orth proj}
Let $V,W\subset X$ be nonzero subspaces, $U\subset V$ be any subspace and $P=\Proj_W$. Then $U$ is principal $\Leftrightarrow$ $P(U)\perp P(U^\perp \cap V)$.\SELF{Usa p/\cref{pr:principal partition}} 
\end{proposition}
\begin{proof}
\emph{($\Rightarrow$)} Follows from \eqref{eq:Pei}.  
\emph{($\Leftarrow$)} By lemmas \ref{pr:orthogonal subspace principal} and \ref{pr:principal subspaces}\emph{\ref{it:complement principal}}, we can assume $U$, $U^\perp \cap V$, $P(U)$ and $P(U^\perp \cap V)$ are not $\{0\}$. As $P(U)\perp P(U^\perp \cap V)$ implies $U\perp P(U^\perp \cap V)$ and $ P(U)\perp U^\perp \cap V$, from associated principal bases of $U$ and $P(U)$, and of $U^\perp \cap V$ and $P(U^\perp \cap V)$, we form principal bases for $V$ and $P(V)$.
\SELF{by \eqref{eq:inner ei fj}}
\end{proof}

\begin{proposition}\label{pr:principal blades proj orth}
Let $V,W\subset X$ be nonzero subspaces, $V_1,V_2\subset V$ be distinct $r$-dimensional  coprincipal subspaces w.r.t. $W$, and $P=\Proj_W$. Then: 
\begin{enumerate}[i)]
\item $V_1\pperp V_2$.
\item If $V_1\not\pperp W$ then $P(V_1)\pperp P(V_2)$.
\item $\inner{\nu_1,\nu_2}=0$ and $\inner{P\nu_1,P\nu_2}=0$ for any $\nu_1\in\Lambda^r V_1$, $\nu_2\in\Lambda^r V_2$. \SELF{Usa p/\cref{pr:identity Theta principal subspaces}.}
\end{enumerate}
\end{proposition}
\begin{proof}
\emph{(i)} $V_1$ has an element $e$ of the principal basis which $V_2$ does not. 
\emph{(ii)} If $V_1\not\pperp W$ then $Pe\neq 0$, and by \cref{pr:principal orth proj} $P(\Span(e))\perp P(V_2)$.
\emph{(iii)} By \cref{pr:partial orth Lambda orth}, $\Lambda^r V_1 \perp \Lambda^r V_2$ and, if $V_1\not\pperp W$, $\Lambda^r P(V_1) \perp \Lambda^r P(V_2)$. If $V_1 \pperp W$ then $P\nu_1=0$, by \cref{pr:partial orthogonality}\emph{\ref{it:Pnu represents PV}}. 
\end{proof}

\begin{theorem}\label{pr:identity Theta principal subspaces}
Given nonzero
\SELF{Não precisa, é só pra facilitar} 
subspaces $V,W\subset X$ and $U\subset V$, let $r=\dim U$ and $p=\dim V$. Then
\[ \cos^2\Theta_{U,W} = \sum_{I\in\I^p_r} \cos^2\Theta_{U,V_I} \cdot \cos^2\Theta_{V_I,W}, \]
where the $V_I$'s are coordinate $r$-subspaces of a principal basis $\beta$ of $V$ w.r.t. $W$.
\end{theorem}
\begin{proof}
The coordinate $r$-blades $\nu_I\in\Lambda^r V_I$ of $\beta$ form an orthonormal basis of $\Lambda^r V$, so $P\mu = \sum_I \inner{\nu_I,\mu} P\nu_I$ for an unit $\mu\in\Lambda^r U$ and $P=\Proj_W$. 
By \cref{pr:principal blades proj orth} the $P\nu_I$'s are mutually orthogonal, so $\|P\mu\|^2 = \sum_I |\inner{\mu,\nu_I}|^2 \|P\nu_I\|^2$. 
The result follows from \cref{pr:products}\emph{\ref{it:Theta inner blades}} and \eqref{eq:norm projection blade}. 
\end{proof}

By \cref{pr:Grassmann coordinate} $\sum_{I\in\I^p_r} \cos^2\Theta_{U,V_I} =1$, so this result means $\cos^2\Theta_{U,W}$ is a weighted average of the $\cos^2\Theta_{V_I,W}$'s, with weights given by the $\cos^2\Theta_{U,V_I}$'s.

\begin{figure}
	\centering
	\includegraphics[width=0.7\linewidth]{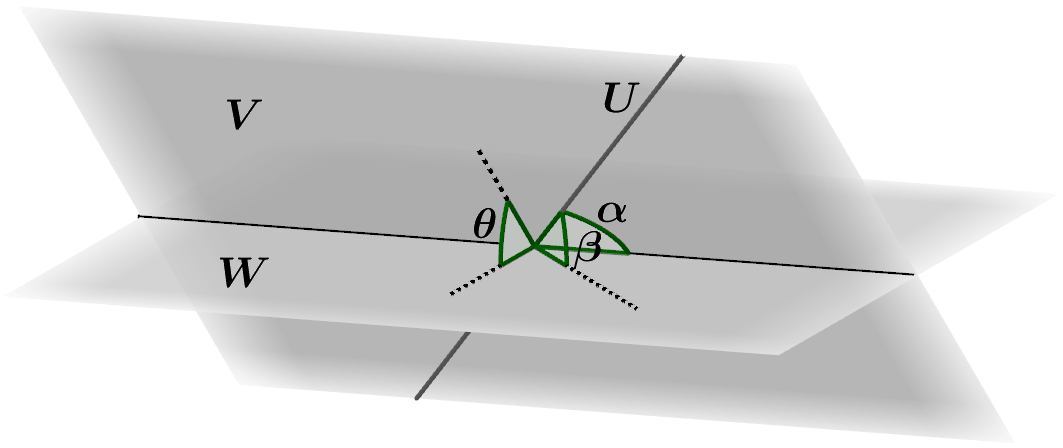}
	\caption{$\cos^2\beta = \cos^2\alpha +\sin^2\alpha\cos^2\theta$}
	\label{fig:angles principal lines}
\end{figure}

\begin{example}
	Given planes $V,W\subset\R^3$ and a line $U\subset V$, let $\alpha=\Theta_{U,V\cap W}$, $\beta=\Theta_{U,W}$ and $\theta=\Theta_{V,W}$ (\cref{fig:angles principal lines}). 
	As $V\cap W$ and $(V\cap W)^\perp\cap V$ are principal lines of $V$ w.r.t. $W$, we have  $\cos^2\beta = \cos^2\alpha\cdot 1 +\sin^2\alpha\cdot\cos^2\theta$. 
\end{example}

\begin{definition}
A partition $V=\bigoplus_i V_i$ is \emph{principal} w.r.t. $W$ if the $V_i$'s are coprincipal subspaces of $V$ w.r.t. $W$.
\end{definition}

Any principal partition is an orthogonal partition. Note that some subspaces of a partition can be $\{0\}$.

\begin{proposition}\label{pr:principal partition}
Let $V,W\subset X$ be nonzero subspaces, $P=\Proj_W$, and $V=\bigoplus_i V_i$ be  an orthogonal partition. The following are equivalent:\SELF{Usa p/ \cref{pr:converse Theta partition}}
\begin{enumerate}[i)]
\item $V=\bigoplus_i V_i$ is a principal partition w.r.t. $W$.\label{it:sum Vi principal}
\item $P(V)=\bigoplus_i P(V_i)$ is a principal partition w.r.t. $V$.\label{it:sum PVi principal}
\item $P(V)=\bigoplus_i P(V_i)$ is an orthogonal partition.\label{it:sum PVi orthogonal}
\end{enumerate}
\end{proposition}
\begin{proof}
\emph{(i\,$\Rightarrow$\,ii)} The $P(V_i)$'s are pairwise disjoint by \cref{pr:principal orth proj}, and coprincipal by \cref{pr:principal subspaces}\emph{\ref{it:P(U) principal}}. 
\emph{(ii\,$\Rightarrow$\,iii)} Immediate.
\emph{(iii\,$\Rightarrow$\,i)} As the $P(V_i)$'s are mutually orthogonal, $V_i \perp P(V_j)$ if $i\neq j$. As the $V_i$'s are also mutually orthogonal, by \eqref{eq:inner ei fj} the union of principal bases of theirs w.r.t. $W$ gives a principal basis of $V$.
\end{proof}

In \cite{Mandolesi_Grassmann} we got $\cos \Theta_{V,W} = \prod_i \cos\Theta_{V_i,W}$ for a principal partition $V=\bigoplus_i V_i$.
We now generalize this for orthogonal partitions and get a partial converse.


\begin{theorem}\label{pr:Theta direct sum}
Let $V_1,V_2,W\subset X$ be subspaces and $P=\Proj_W$. If $V_1\perp V_2$,
\begin{equation*}
\cos \Theta_{V_1\oplus V_2,W} = \cos \Theta_{V_1,W}\cdot \cos \Theta_{V_2,W}\cdot \cos \Theta^\perp_{P(V_1),P(V_2)}.
\end{equation*}
\end{theorem}
\begin{proof}
By propositions \ref{pr:partial orthogonality}\emph{\ref{pr:subspace not pperp}} and \ref{pr:properties Grassmann}\emph{\ref{it:Theta pi2}}, we can assume $V_1\not\pperp W$ and $V_2\not\pperp W$.
As $V_1\perp V_2$, if $\nu_1$ and $\nu_2$ are unit blades representing them, $\nu_1\wedge\nu_2$ is an unit blade representing $V_1\oplus V_2$. 
By \eqref{eq:norm projection blade}, $\cos \Theta_{V_1\oplus V_2,W} = \|P(\nu_1\wedge\nu_2)\| = \|(P\nu_1)\wedge(P\nu_2)\|$, and the result follows from propositions \ref{pr:partial orthogonality}\emph{\ref{it:Pnu represents PV}} and \ref{pr:products}\emph{\ref{it:exterior product}}.
\end{proof}

\begin{corollary}\label{pr:Theta orthog partition}
Let $V,W\subset X$ be subspaces and $P=\Proj_W$. For an orthogonal partition $V=\bigoplus_{i=1}^k V_i$,
\begin{equation*}
\cos \Theta_{V,W} = \prod_{i=1}^k \cos \Theta_{V_i,W} \cdot \prod_{i=1}^{k-1}\cos \Theta^\perp_{P(V_i),P(V_{i+1}\oplus\ldots\oplus V_k)}.
\end{equation*}
\end{corollary}

\begin{proposition}\label{pr:converse Theta partition}
For nonzero subspaces $V,W\subset X$ with $V\not\pperp W$, a partition  $V=\bigoplus_i V_i$ is principal w.r.t. $W$ if, and only if, it is orthogonal and $\cos \Theta_{V,W} = \prod_i \cos \Theta_{V_i,W}$.
\end{proposition}
\begin{proof}
\Cref{pr:Theta orthog partition}  and propositions \ref{pr:complementary simple cases}\emph{\ref{it:Theta perp 0}} and \ref{pr:principal partition} give the converse.
\end{proof}


\providecommand{\bysame}{\leavevmode\hbox to3em{\hrulefill}\thinspace}
\providecommand{\MR}{\relax\ifhmode\unskip\space\fi MR }
\providecommand{\MRhref}[2]{%
	\href{http://www.ams.org/mathscinet-getitem?mr=#1}{#2}
}
\providecommand{\href}[2]{#2}

\end{document}